\documentclass[letterpaper]{article}

\usepackage[english]{babel}

\usepackage[T1]{fontenc}

\usepackage[a4paper,top=3cm,bottom=3cm,left=3cm,right=3cm,marginparwidth=1.75cm]{geometry}

\usepackage{graphicx}
\usepackage{url}
\usepackage{amsmath,amsfonts,amssymb,color}
\usepackage[colorlinks]{hyperref}
\usepackage{mathptmx}     
\usepackage{amsmath,amsfonts,amsthm,amssymb,color}

\usepackage[shortlabels]{enumitem}

\newtheorem{theorem}{Theorem}[section]
\newtheorem{corollary}[theorem]{Corollary}
\newtheorem{lemma}[theorem]{Lemma}

\newtheorem{conjecture}[theorem]{Conjecture}

\theoremstyle{definition}
\newtheorem{definition}[theorem]{Definition}

\begin{document}

\title{Diagonal form of the Varchenko matrices for oriented matroids
}

\author{Assylbek Olzhabayev, YiYu Zhang
}

\maketitle

\begin{abstract}

The construction of the Varchenko matrix for hyperplane arrangements, first introduced by Alexandre Varchenko, extends naturally to oriented matroids. In this paper, we generalize the theorem of Gao and Zhang by proving that the Varchenko matrix of an oriented matroid has a diagonal form if and only if the pseudohyperplane arrangement corresponding to the oriented matroid is in semigeneral position, i.e. it does not contain a degeneracy. 

Furthermore, we show that the Varchenko matrix of a pseudoline arrangement has a block diagonal form.  This also provides an alternative combinatorial proof for the Varchenko matrix determinant formula in dimension two.

\end{abstract}

\section{Introduction}
Varchenko \cite{varchenko1993bilinear} defined the Varchenko matrix associated with any \textit{hyperplane arrangement}. The first studies on the existence of a diagonal form of the Varchenko matrix were conducted by Denham and Hanlon \cite{denham1997smith}, followed by Cai and Mu \cite{cai2016smith} and, finally, by Gao and Zhang \cite{gao2018diagonal}, who proved that the \textit{Varchenko matrix} has a diagonal form if and only if the corresponding arrangement is in \textit{semigeneral position}, i.e. contains no degeneracy. Recently, Hochst{\"a}ttler and Welker \cite{hochstattler2018varchenko} defined the Varchenko matrix for oriented matroids, which can be viewed as a generalization of hyperplane arrangements, and computed its determinant. 

\

In this paper, we note a straightforward generalization of the main theorem in \cite{gao2018diagonal} and its proof to oriented matroids via the topological representation by \textit{pseudohyperplane arrangements}, which determines the condition for the Varchenko matrix of an oriented matroid to have a diagonal form.
\begin{theorem} \label{gen}
For $\mathcal{M}$ an oriented matroid, let $\mathcal{A}=\{H_1,\ldots,H_N\}$ be its corresponding pseudohyperplane arrangement. Assign an indeterminate $x_a$ to each $H_a$, $a\in I=\{1,\ldots,N\}$. Then the Varchenko matrix $V$ associated with $\mathcal{A}$ has a diagonal form over $\mathbb{Z}[x_1,\ldots,x_N]$ if and only if $\mathcal{A}$ is in semigeneral position. In that case, the diagonal entries of the diagonal form of $V$ are exactly the products $$\prod_{a\in B}(1-x^2_a)$$  ranging over all $B\subseteq I$ such that $H_B\in L(\mathcal{A}).$
\end{theorem}

Then we show that the Varchenko matrix of a pseudoline arrangement has a \textit{block diagonal form}, i.e., it can be represented by a block matrix where every non-zero block is on the main diagonal.  For each degenerate intersection formed by $n$ pseudolines, there is a block of size $n-1$ on the diagonal called a \textit{leftover matrix}, whose entries are determined explicitly and combinatorially. Each pseudoline, nondegenerate intersection point, and the ambient plane corresponds to one diagonal entry, as in the case of semigeneral hyperplane arrangements in \cite{gao2018diagonal}. More precisely, we prove the following theorem:

\begin{theorem}\label{theorem45}
Let $\mathcal{A}=\{l_1,\ldots, l_n\}$ be a pseudoline arrangement. Assign to each pseudoline $l_a$ the indeterminate $x_a$. Then the Varchenko matrix of $\mathcal{A}$ has a block diagonal form over $\mathbb{Z}[x_1,\ldots,x_n]$. The block entries of size greater than one are precisely the leftover matrices of degenerate intersection points, and the non-degenerate elements of the intersection poset are in bijection with the rest of the diagonal entries, which have the form $$\prod_{a \in B}(1-x_a^2)$$ ranging over all $B\subseteq \{1,\ldots,n\}$ such that $\bigcap_{a\in B} l_a$ is an element of the intersection poset $L(\mathcal{A})$.
\end{theorem}

\begin{figure}
    \centering
    \includegraphics[scale=0.7]{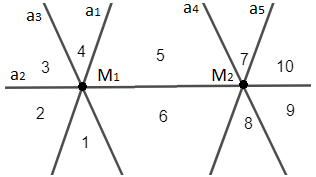}
    \caption{Arrangement with 2 degenerate points.}
    \label{fig:2deg}
\end{figure}

\

For example, the Varchenko matrix of the arrangement in Figure~\ref{fig:2deg} has the following block diagonal form. The $2\times 2$ matrices on the main diagonal come from the two degenerate intersections $M_1, M_2$ respectively, each of which are formed by 3 pseudolines.
\begin{equation*}
\begin{smallmatrix}
1&0&0&0&0&0&0&0&0&0\\
0&(1-a_1^2)&0&0&0&0&0&0&0&0\\
0&0&(1-a_2^2)&0&0&0&0&0&0&0\\
0&0&0&(1-a_3^2)&0&0&0&0&0&0\\
0&0&0&0&(1-a_1^2)(1-a_2^2a_3^2)&a_2(1-a_1^2)(1-a_3^2)&0&0&0&0\\
0&0&0&0&a_2(1-a_1^2)(1-a_3^2)&(1-a_3^2)(1-a_1^2a_2^2)&0&0&0&0\\
0&0&0&0&0&0&(1-a_4^2)&0&0&0\\
0&0&0&0&0&0&0&(1-a_5^2)&0&0\\
0&0&0&0&0&0&0&0&(1-a_4^2)(1-a_2^2a_5^2)&a_2(1-a_4^2)(1-a_5^2)\\
0&0&0&0&0&0&0&0&a_2(1-a_4^2)(1-a_5^2)&(1-a_2^2a_4^2)(1-a_5^2)
\end{smallmatrix}
\end{equation*}

This result also serves as an alternative elementary proof of the Varchenko matrix determinant formula for pseudoline arrangements.

\

The organization of this paper is as follows. The basics of hyperplane arrangements, oriented matroids, and the topological representation theorem are outlined in Section \ref{section 2}. In Section \ref{section 3}, we generalize the theorem of Gao and Zhang in \cite{gao2018diagonal} regarding the existence of a diagonal form of the Varchenko matrix of an oriented matroid. Then we consider the "next best thing", i.e., the block diagonal form for the Varchenko matrices of of pseudoline arrangements that have degenerate intersections. In Section \ref{section 4}, we compute a block diagonal form for the Varchenko matrix of a pseudoline arrangement with one degenerate point formed by $n$ pseudolines by partitioning the arrangement into two subarrangements and "resolving" the degenerate point. We generalize the method in Section \ref{section 5} and construct a block diagonal form for pesudoline arrangements with arbitrary degeneracies.

\section{Preliminaries}\label{section 2}
In this section, we go over the basic theories of hyperplane arrangements and oriented matroids, as well as the topological representation theorem of oriented matroids.

\subsection{Hyperplane arrangements}
A \textit{hyperplane arrangement} is a finite set of  hyperplanes in $\mathbb{R}^d$. Given an arrangement $\mathcal{A}$ with $n$ hyperplanes, we assign the indeterminate $x_a$ for each hyperplane $H_a$. For any subset $B \subseteq \{1,2,\ldots,n\}$, denote by $H_B$ the intersection $\bigcap_{a \in B} H_a$.

Let $L(\mathcal{A})$ be the set of all
nonempty intersections of hyperplanes in $\mathcal{A}$, including $\mathbb{R}^d$ itself as the intersection over the empty set. Define $x \leqslant y$ in $L(A)$ if $x \supseteq	 y$ (as subsets of $\mathbb{R}^d$). In other words, $L(\mathcal{A})$ is partially ordered by reverse inclusion. We call $L(\mathcal{A})$ the \textit{intersection poset}
of A. 
  
We define a \textit{region} $R$ of $\mathcal{A}$ to be a connected component of the complement of the union of all hyperplanes $\bigcup H_a$ in $\mathbb{R}^d$. Denote by $\mathcal{R}(\mathcal{A})$ the set of regions of $\mathcal{A}$ and $r(\mathcal{A})$ the number of regions. We will say that a hyperplane \textit{separates} two regions $R_m$ and $R_n$ if the regions are in different subspaces upon hyperplane. Denote $sep(R_m,R_n):=\{H_a \in \mathcal{A}:\ H_a$ separates $R_m$ and $R_n\}$.

We say that an arrangement $\mathcal{A}$ is a \textit{general} arrangement in $\mathbb{R}^d$ if for any subset $B \subseteq \{1,2,\ldots,n\}$, the cardinality $|B| \leqslant d$ implies that $\dim(H_B)=d-|B|$, while $|B|>d$ implies that $H_B=\emptyset$.
If for all $B \subseteq \{1,2,\ldots,n\}$ with $H_B \neq \emptyset$, we have $\dim(H_B)=d-|B|$, then $\mathcal{A}$ is called \textit{semigeneral} in $\mathbb{R}^d$.

The \textit{Varchenko matrix} $V(\mathcal{A})=[V_{ij}]$ of a hyperplane arrangement $\mathcal{A}$ is the $r(\mathcal{A})\times r(\mathcal{A})$ matrix with rows and columns indexed by $\mathcal{R}(\mathcal{A})$ and entries
\begin{equation*}
    V_{ij}=\prod_{H_a \in sep(R_i, R_j)}^{} x_a.
\end{equation*}

\subsection{Oriented matroids}

Next, we review some definitions and constructions from the theory of oriented matroids.

\begin{definition}
A \textit{signed set} $X$ is a set $X$ together with a partition $(X^+, X^-)$ of $X$ into two
disjoint subsets: $X^+$, the set of positive elements of $X$, and $X^-$, its set of
negative elements.  
\end{definition}
The set $\underline{X} = X^+ \cup X^-$ is called the \textit{support} of X.

\begin{definition}(cf. \cite[3.2.1]{bjorner1999oriented})
An \textit{oriented matroid} is a pair of sets $(E,\mathcal{C})$, where $\mathcal{C}$ is the set of some signed subsets of $E$, called \textit{signed circuits}, such that
\begin{enumerate}
    \item $\emptyset \notin	\mathcal{C}$,
    \item if $X \in \mathcal{C}$, then $-X \in \mathcal{C}$,
    \item for all $X, Y \in \mathcal{C}$, if $\underline{X}	\subseteq \underline{Y}$, then $X=Y$ or $X=-Y$,
    \item for all $X, Y \in \mathcal{C}$ such that $X \neq -Y$, and $e \in X^+ \cap Y^-$ there is a $Z \in \mathcal{C}$ such that
$Z^+ \subseteq (X^+ \cup Y^+)\backslash\{e\}$ and
$Z^- \subseteq (X^- \cup Y^-)\backslash\{e\}$.
\end{enumerate}
\end{definition}

If a signed circuit $e\in E$ satisfies $(e,\emptyset)\in\mathcal{C}$, then it is called a \textit{loop}.

\

In order to define the Varchenko matrix of an oriented matroid, it is convenient to give an equivalent definition via \textit{topes}. (cf. \cite[3.8.3]{bjorner1999oriented}.)
\begin{definition}\label{deftopes}
An \textit{oriented matroid} is a pair of sets $(E,\mathcal{T})$, where $\mathcal{T}$ is the set of some signed subsets of $E$, called \textit{topes}, such that
\begin{enumerate}
    \item $\mathcal{T} \neq \emptyset$, and for all $X, Y \in \mathcal{T}$ we have $\underline{X}=\underline{Y}$,
    \item if $X \in \mathcal{T}$, then $-X \in \mathcal{T}$,
\item for all $A, B \subseteq E$ with $A \cap B \neq \emptyset$ we have $(\mathcal{T}\backslash A)/B=(\mathcal{T}\backslash B)/A$, where $A/B$ denotes the contraction of $A$ to $B$.
\end{enumerate}
\end{definition}

All loops of an oriented matroid form a \textit{zero set} $E_0$ such that for every $X \in \mathcal{T}$ $\underline{X}\cap E_0=\emptyset$.

For two topes $P$ and $Q$ we define the separation set $Sep(P,Q)=\{e \in E | P_e=-Q_e \neq 0\}$. Now we are ready to define the Varchenko matrix for oriented matroids.
\begin{definition}
For the oriented matroid $\mathcal{L}$ we assign to each $e \in E$ the indeterminate $x_e$. The \textit{Varchenko matrix} is a matrix $V=[V_{ij}]$ with the number of rows and columns equal to the number of topes of $\mathcal{L}$ such that for any two topes $T_i$ and $T_j$ (not necessarily distinct)
\begin{equation*}
V_{ij}=\prod_{e \in Sep(T_i,T_j)}^{} x_e.
\end{equation*}
\end{definition}

As we will see later (Theorem~\ref{lemma:trt}), loop-free oriented matroids are closely connected with configurations of topological objects called \textit{pseudospheres}. 

\begin{definition}
A subset $S$ of $S^d$ is called a \textit{pseudosphere} if $S = h(S^{d-1})$ for some homeomorphism
$h : S^d \rightarrow S^d$, where $S^{d-1}$ is a $(d-1)$-dimensional sphere and a subset of $S^d$. A pseudosphere $S$ has two sides $S^+$ and $S^-$, called hemispheres.
\end{definition}

\begin{definition}
An \textit{arrangement of pseudospheres} $\mathcal{A}$ is a finite set of pseudospheres $\{S_e |e\in \{1,2,...\}$) in the unit sphere $S^d$ such that every non-empty intersection of a subset of pseudosheres is a pseudosphere. Furthermore, for $e \in E$ and $A	\subset E$, such that $S_A \nsubseteq S_e$, the intersection $S_A \cap S_e$ is a pseudosphere with sides $S_A \cap S_e^+$ and $S_A \cap S_e^-$ and every pseudosphere is symmetric about the center of $S^d$. 
\end{definition}
The intersection poset of a pseudospehere arrangement is defined similarly to the intersection poset of a hyperplane arrangement. 
If $\mathcal{A}$ is a signed arrangement of pseudospheres, let $\mathcal{C}(\mathcal{A})$ be the set of signs of the regions.

Similarly, we  can define pseudohyperplanes as well.
A subspace $R$ of $\mathbb{R}^d$ will be called a \textit{pseudohyperplane} if $R = h(\mathbb{R}^{d-1})$ for some homeomorphism
$h : \mathbb{R}^d \rightarrow \mathbb{R}^d$, where $\mathbb{R}^{d-1}$ is a subset of $\mathbb{R}^d$. A pseudohyperplane $R$ has two sides  $R^+$ and $R^-$.
 A pseudohyperplane arrangement and its Varchenko matrix is defined similarly as in the case of hyperplane arrangements.

\subsection{Topological Representation of Oriented Matroids}
The following theorem by Folkman and Lawrence gives us a concrete description of all loop-free oriented matroids in terms of pseudosphere arrangements. See Bj{\"o}rner et al. \cite{bjorner1999oriented} for instance.
\begin{theorem}\label{lemma:trt}
\begin{enumerate}[(a)]
    \item If $\mathcal{A}$ is a signed arrangement of pseudospheres in $S^d$, then $\mathcal{C}(A)$ is a set of circuits of some loop-free oriented matroid.
    \item If $(E,\mathcal{C})$ is a loop-free oriented matroid, then there is a pseudosphere arrangement $\mathcal{A}$ with $\mathcal{C}(A)=\mathcal{C}$ in $S^d$ for some $d$.
    \item $\mathcal{C}(A) = \mathcal{C}(A')$ for two signed arrangements $A$ and $A'$ in $S^d$ if and only if
$A' = h(A)$ for some self-homeomorphism $h$ of $S^d$.
\end{enumerate}

\end{theorem}

However, pseudospheres are not easy to work with, so we consider pseudohyperplanes instead. Specifically, we use the  stereographic projection from $S^d$ to $\mathbb{R}^d$  to obtain segments of pseudohyperplanes out of pseudohemispheres. Then we expand the pseudohyperplanes so that no new intersections are introduced, and we get a pseudohyperplane arrangement. 
Due to the symmetric property of pseudospheres, a pseudosphere arrangement is uniquely determined by the pseudohemisphere arrangement, that lies on the southern hemisphere of $S^d$. This preserves the intersection poset, thus we can consider pseudohyperplane arrangements instead of pseudosphere arrangements. Notice that the two definitions of Varchenko matrix for pseudohyperplane arrangements and oriented matroids agree with each other.

\

For an arbitrary oriented matroid $\mathcal{M}:=(E,\mathcal{T})$ with $\mathcal{T}$ the set of topes, the Topological Representation Theorem does not apply directly. Nonetheless, its Varchenko matrix can be understood via pseudohyperplane arrangements for the following reason. We consider the support $S$ of the topes. It is easy to verify that $\mathcal{M}':=(S,\mathcal{T})$ is a matroid by checking the axioms given in Definition~\ref{deftopes}. Moreover, $\mathcal{M}'$ is loop-free, because the zero set of its topes is empty. Finally, we notice that $\mathcal{M}$ and $\mathcal{M}'$ have the same Varchenko matrix because they have the same set of topes with the same separating sets. For the remainder of the paper, we assume that we are working with loop-free oriented matroids.

\section{Diagonal form of the Varchenko matrix of a semigeneral pseudohyperplane arrangement} \label{section 3}

In \cite{gao2018diagonal}, Gao and Zhang proved that the Varchenko matrix of a hyperplane arrangement $\mathcal{A}$ is diagonalizble over the integer polynomial ring if and only if $\mathcal{A}$ is semigeneral. This result and its proof can be easily generalized to pseudohyperplane arrangements and thus to oriented matroids.

\begin{theorem}\label{lemma:the}
Consider a pseudohyperplane arrangement $\mathcal{A}=\{H_1,...,H_n\}$ in $\mathbb{R}^d$. Assign an indeterminate $x_a$ to each pseudohyperplane $H_a$, where $a \in I=\{1,...,n\}$. Then the Varchenko matrix associated with $\mathcal{A}$ is diagonalizable over $\mathbb{Z}[x_1,x_2,...,x_n]$ if and only if $\mathcal{A}$ is in semigeneral position. In that case, the diagonal entries of the diagonal form of $V$ are equal to $$\prod_{a \in B}^{} (1-x_a)$$ ranging over all $B \subseteq I$ such that $H_B \in L(\mathcal{A})$.
\end{theorem}

Note that this is not a trivial generalization, because there are pseudohyperplane arrangements which are not homeomorphic to hyperplane arrangements, such as Figure~\ref{fig:my_label1},  described in \cite[Chapter 5]{felsner2017pseudoline}. 
\begin{figure}[h]
    \centering
    \includegraphics[scale=0.5]{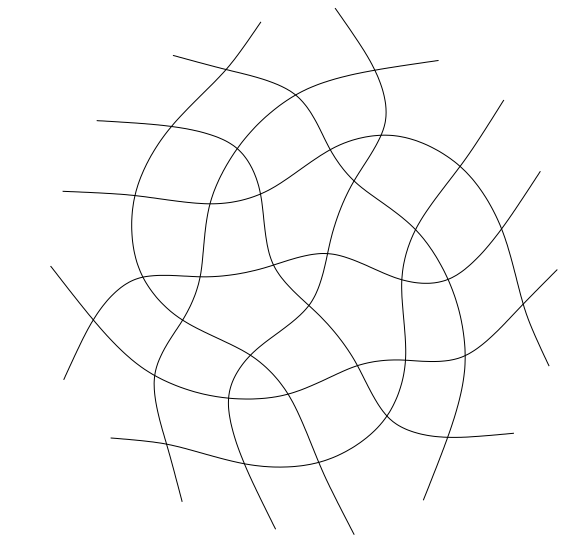}
    \caption{Non-stretchable pseudoline arrangement in $\mathbb{R}^2$.}
    \label{fig:my_label1}
\end{figure}

\

The proof of the "if" part of Theorem~\ref{lemma:the} consists of two steps: first, we prove the existence of a numbering of the regions such that every time we index a new region, we encompass exactly one element in the intersection poset. The second step is to find a combinatorial interpretation of the entries after certain row and column operations, and use it to show that we can reduce the matrix to its diagonal form. Now we want to check that the same proof works for a semigeneral pseudohyperplane arrangement.

\begin{definition}
\begin{enumerate}
    \item A set of regions $\mathcal{B} \subset \mathcal{R}(\mathcal{A})$ \textit{encompasses} a point $x \in \mathbb{R}^d$ if the interior of the closure of the union of these regions contains $x$.
    \item A set of regions $\mathcal{B} \subset \mathcal{R}(\mathcal{A})$ \textit{encompasses} an element $M\in L(\mathcal{A})$ if there exists a point $x\in M$ such that $\mathcal{B}$ encompasses $x$.
\end{enumerate}
\end{definition}

The next lemma asserts the existence of a "good" numbering of the regions of the arrangement with certain properties.

\begin{lemma}\label{lemma:5}

Consider a pseudohyperplane arrangement $\mathcal{A}$ with $s$ regions. Then we can number the regions by $I=\{1,2,...,s\}$ in increasing order such that any time we index a new region, we encompass exactly one new element of the intersection poset and the following properties remain true for all $n	\in I$:
\begin{enumerate}[(a)]
    \item The  interior of the closure of $\bigcup_{1\leqslant	m\leqslant	n}R_m$ is connected.
    \item For all $M \in L(\mathcal{A})$, the subset $\{x:x \in M, x \hspace{5 pt} is\hspace{5 pt} already\hspace{5 pt} encompassed\}\subseteq	M$ is connected.
    \item If $R_n$ is first to encompass $M=H_B$, where $B \subseteq I$, then $R_n$ is the first indexed region in the cone formed by all $H_a$, $a \in B$ that contains $R_n$.
\end{enumerate}
\end{lemma}

The proof of this lemma uses an inductive approach similar to that of Gao and Zhang \cite[Lemma 5]{gao2018diagonal}, because properties 1 and 2 of the lemma concern only a small neighborhood of the new region, which is always stretchable to a hyperplane arrangement, while the proof of property 3 relies solely on properties 1 and 2.

\

For the second part of the proof, we define $V^{(k)}$ inductively in the following way: we set $V^{(0)}=V(\mathcal{A})$ and obtain $V^{(k)}$ from $V^{(k-1)}$ by eliminating the entries $V_{j,k-1}^{(k-1)}$, $j >k-1$ via subtracting multiples of the $k-1$-th row of $V^{(k-1)}$ from the rows below in $V^{(k-1)}$, and then eliminating the entries $V_{k-1,j}^{(k-1)}$, $j >k-1$ in an analogous way. The proof of Theorem~\ref{lemma:the} is now reduced to the following lemma, stated and proved by Gao and Zhang \cite{gao2018diagonal}.
\begin{lemma}\label{lemma:6}
The entries of the matrix $V^{(k)}$ are as follows:
\begin{enumerate}[(a)]
    \item The diagonal entry $V_{k,k}^{(k)}=\prod_{a \in x_k}^{} (1-x_a)^2$ for all $a$ such that $H_a$ passes through the element of the intersection poset that is newly encompassed by $R_k$;
    \item For all $m \neq n \leqslant k$, $V_{m,n}^{(k)}=0$; for all $m \leqslant k$ $V_{m,m}^{(k)}=V_{m,m}^{(m)}$ ;
        \item If at least one of $m$, $n$ is greater than $k$, then $V_{m,n}^{(k)}=V_{m,n} \cdot \phi(\prod_{i=1}^{k} (1-l_i^2(m,n))$ , where $\phi$ is the function on polynomial that replaces all exponents of indeterminates that are greater or equal to three by two, and $l_i(m,n)$ is the product of the weights of all pseudohyperplanes that separate simultaneously  $(R_i,R_m)$ and $(R_i,R_n)$.
\end{enumerate}
\end{lemma}
The proof for the hyperplane case is applicable here because it does not use any global geometrical properties of hyperplanes and rely only on the Lemma~\ref{lemma:5}, so is the proof for the "only if" part of the theorem.

\section{Block diagonal form of arrangements with one degeneracy}\label{section 4}
In this section, we construct a block diagonal form of the Varchenko matrix associated to a pseudoline arrangement with one degeneracy. This can be viewed as a generalization of \ref{lemma:the}, since a diagonal form is not achievable for degenerate arrangements. It also provides an entirely elementary way to compute the determinant of the Varchenko matrix.

\subsection{The $n$-basic arrangements}\label{subsection:1}
First, we look at the simplest arrangement with one degeneracy, i.e., $n$ lines intersecting in one point. We call this arrangement the \textit{n-basic degeneracy arrangement}. We note that the $n$-basic degeneracy case can be generalized to the case where $n$ pseudohyperplanes intersect in exactly one subspace of codimension two using the same arguments.  

We apply the diagonalization operation, described in Section~\ref{section 3}, $n+1$ times, and we get an $(n-1)\times (n-1)$ symmetric submatrix in the lower-right. We call it the \textit{leftover matrix}, denoted by $L^n=[L_{i,j}^n]$.

By Lemma~\ref{lemma:6}, the first $n+1$ entries of the Varchenko matrix are exactly $1$, $1-x_1^2$, \ldots, $1-x_n^2$. Moreover, it gives us an explicit formula for the entries of the leftover matrix $L_{i,i+j}^n=x_{i+1}x_{i+2}\cdots x_{i+j}(1-x_1^2x_2^2\cdots x_i^2)(1-x_{i+j+1}^2x_{i+j+2}^2\cdots  x_n^2)$ for any $i \geq 1$ and $j\geq 0$.

\

Now we compute the determinant  $\det(L^n)$.
\begin{lemma}
The determinant of $L^n$ is equal to $(1-x_1^2)(1-x_2^2)\cdots (1-x_n^2)(1-x_1^2x_2^2\cdots x_n^2)^{n-2}$.
\end{lemma}
\begin{proof}
We proceed by induction on $n$. 

Base case: $n=3$. We know that
   $$L^3=\begin{pmatrix}
(1-x_1^2)(1-x_2^2x_3^2) & x_2(1-x_1^2)(1-x_3^2)\\ 
x_2(1-x_1^2)(1-x_3^2) & (1-x_1^2x_2^2)(1-x_3^2)
\end{pmatrix},$$ so
\begin{align*}
    \det(L^3)&=(1-x_1^2)(1-x_2^2x_3^2)\cdot(1-x_1^2x_2^2)(1-x_3^2)-x_2(1-x_1^2)(1-x_3^2)\cdot x_2(1-x_1^2)(1-x_3^2)\\
    &=(1-x_1^2)(1-x_2^2)(1-x_3^2)(1-x_1^2x_2^2x_3^2).
\end{align*} 
This completes the base case.

 Induction step: suppose that the claim holds for $n-1$. We prove it for $n$. First, we subtract the second row, multiplied by $\frac{x_2(1-x_1^2)}{(1-x_1^2x_2^2)}$, from the first row. Denote the resulting matrix by $L^{n'}=[L^{n'}_{i,j}]$. Notice that now every entry in the second row except $L_{1,1}$ is zero, because $L_{1,k}=x_2x_3\cdots x_k(1-x_1^2)(1-x_{k+1}^2\cdots x_n^2)$ and $L_{2,k}=x_3x_4\cdots x_k(1-x_1^2x_2^2)(1-x_{k+1}^2\cdots x_n^2)$ for $k \geq 2$. Since $L_{1,1}=(1-x_1^2)(1-x_2^2x_3^2\cdots x_n^2)$, after subtraction it becomes 
\begin{align*}
L^{n'}_{1,1} &=(1-x_1^2)(1-x_2^2x_3^2\cdots x_n^2)-x_2(1-x_1^2)(1-x_3^2x_4^2\cdots x_n^2)\cdot \dfrac{x_2(1-x_1^2)}{(1-x_1^2x_2^2)}\\
&=\dfrac{(1-x_1^2)(1-x_2^2)(1-x_1^2x_2^2\cdots x_n^2)}{1-x_1^2x_2^2}.    
\end{align*}
Denote by $L^{n"}$ the matrix formed by elimination of the first row and the first column from $L^n$. Note that $L^{n"}$ is $L^{n-1}$ with indeterminates $x_1 x_2, x_3,\ldots, x_n$, so its determinant is equal to $(1-x_1^2 x_2^2)(1-x_3^2)\cdots(1-x_n^2)(1-x_1^2x_2^2\cdots x_n^2)^{n-3}$. In the first row of $L^{n'}$ every entry except the first one is zero, so
\begin{align*}
    \det(L^{n'})&=L^{n'}_{1,1}\cdot \det(L^{n"})\\
    &=\dfrac{(1-x_1^2)(1-x_2^2)(1-x_1^2x_2^2\cdots x_n^2)}{1-x_1^2x_2^2}\cdot (1-x_1^2x_2^2)(1-x_3^2)\cdots (1-x_n^2)(1-x_1^2x_2^2\cdots x_n^2)^{n-3}\\
    &=(1-x_1^2)(1-x_2^2)\cdots (1-x_n^2)(1-x_1^2x_2^2\cdots x_n^2)^{n-2}.
\end{align*}
 This completes the induction step.
\end{proof}

Thus we conclude that the determinant of Varchenko matrix for the $n$-basic degeneracy arrangement is equal to $(1-x_1^2)(1-x_2^2)\cdots (1-x_n^2)(1-x_1^2x_2^2\cdots x_n^2)^{n-2}$.

\subsection{Arrangements with one degeneracy}

In this section, we generalize the proof to arbitrary pseudoline arrangements in $\mathbb{R}^2$ with one degeneracy point $M$. First, we prove a technical lemma concerning the function $\phi$.
\begin{lemma}\label{lemma:7}
Let $P$, $X_1$, $X_2$,\ldots, $X_n$ be products of some indeterminates from the set $\{x_1,x_2,\ldots,x_n\}$ such that $\gcd(P,X_i)=1$ for each $i \in \{1,2,\ldots,n\}$. Then $$\displaystyle{\phi((1-P^2)\cdot\prod^n_{i=1} (1-X_i^2P^2))=1-P^2}.$$
\end{lemma}
\begin{proof}
Consider some nonempty subset $B \subseteq \{1,2,\ldots,n\}$. Note that if we consider $Q=(1-P^2)\cdot\prod^n_{i=1} (1-X_i^2P^2)$ as a polynomial in $P$, there are exactly two monomials of the form  $\prod_{i\in B} X_i^2 \cdot P^N$ in its expansion, which come from multiplying by the $1$ and the $P^2$ in $1-P^2$. In both cases $N \geq 2$, so when we apply $\phi$ to the expansion of $Q$, these two monomials will become $P^2\cdot\phi(\prod_{i\in B} X_i^2)$ and $-P^2\cdot\phi(\prod_{i\in B} X_i^2)$, thus cancelling each other out. We cannot use the same arguments when $B$ is the empty set, so $$\phi((1-P^2)\cdot\prod^n_{i=1} (1-X_i^2P^2))=1-P^2$$ as desired.
\end{proof}

Now we are ready to state the first main result of this paper.

\begin{theorem}\label{theorem1degeneracy}
Consider the pseudoline arrangement $\mathcal{A}$ with one degeneracy point $M$ with $n$ lines passing through it. Then the corresponding Varchenko matrix has a block-diagonal form consisting of three blocks: the upper-left and lower-right blocks are in diagonal form whose diagonal entries are exactly $$\prod_{a \in B}(1-x_a^2)$$ ranging over all $B\subseteq I$ such that $H_B \in L(\mathcal{A})$; the middle block is the leftover matrix $L^n$ of the $n$-basic degeneracy arrangement.
\end{theorem}

\begin{proof}

The idea of the proof is to split the whole arrangement into two parts by pseudolines and their segments in the arrangement so that the border passes through the degeneracy and there are exactly $n-1$ degeneracy-bordering regions in one of the subarrangements. This can be achieved in the following way: we consider the line $H_a$ that is the closest to the degeneracy (if there are several such lines, consider any of them), but not passing through it. Then we mark the points of intersection of this line and lines passing through the degeneracy and take the outermost marked points. Without loss of generality we assume that neither point is at infinity. Denote these two points $P$ and $Q$. We split the arrangement by the pseudoline formed by segments $PM$ and $QM$ and $H_a$ minus the segment $PQ$, as shown in Figure~\ref{fig:my_label2}.
\begin{figure}[h]
    \centering
    \includegraphics[scale=0.6]{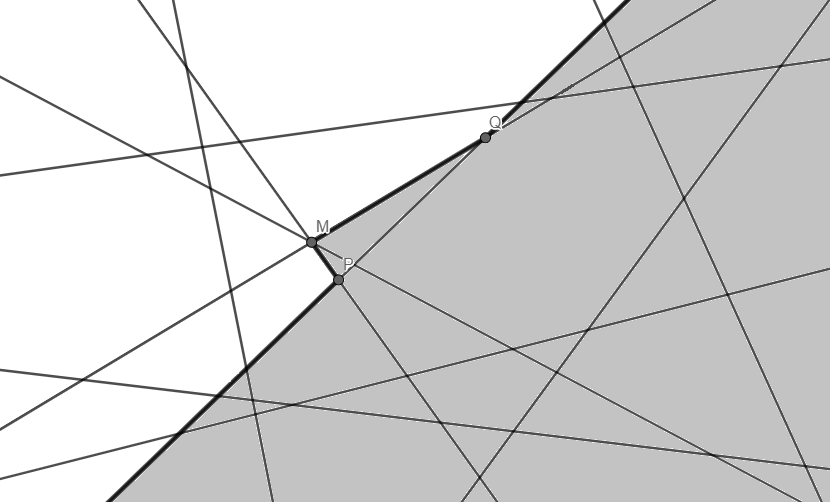}
    \caption{Partition into subarrangements.}
    \label{fig:my_label2}
\end{figure}
Note that we choose this separating line so that there is no line with two disconnected parts in the same subarrangement, so we are not encompassing anything twice.

We do not include the degeneracy point in either subarrangement by cutting out a small enough neighborhood around the degeneracy. We call the subarrangement containing the segment $PQ$ the second, and the other one the first.

Consider the arrangement $\mathcal{A}'$ consisting of $H_a$ and all the lines in the second subarrangement, expanded so that no intersections are introduced. Clearly, $\mathcal{A}'$ is in semigeneral position. We say that a region is an edge region, if it corresponds to the region from the first subarrangement of $\mathcal{A}$ or borders a degeneracy. The proof of Lemma~\ref{lemma:5} actually claims that we can continue doing the good numbering if three properties from the Lemma~\ref{lemma:5} remain true, so we can number first the edge regions in a linear order, i.e. every time we encompass one line. We note that the three properties from Lemma~\ref{lemma:5} remain true. Hence we can continue the numbering using the algorithm as in the proof of Lemma~\ref{lemma:5}.

Now we number the whole arrangement in the following way: we number the first subarrangement in a good way and then number the second subarrangement in a good way so that the regions bordering the degeneracy in the second subarrangement will be indexed first. To obtain the numbering of the second subarrangement, we index the regions bordering degeneracy and number other regions as for $\mathcal{A}'$. 

Now we apply the diagonalization technique introduced in Section~\ref{section 3} to the Varchenko matrix of the whole arrangement. Notice that the entries on the diagonal corresponding to the regions of the first subarrangement are exactly as described in Lemma~\ref{lemma:6}, and the rest of the matrix is not diagonalizable by Theorem \ref{lemma:the}. Denote the submatrix formed by the rows and columns corresponding to the regions of the second subarrangement by $S=[S_{i,j}]$. Notice that the $(n-1)\times (n-1)$ submatrix of $S$ in the upper-left will be exactly $L^n$, because
$$\displaystyle{S_{f,j}=V^{(k)}_{k+f,k+j}=V_{k+f,k+j}\cdot \phi (\prod^k_{i=1} (1-l_i^2(k+f,k+j)))}$$
and if we consider arrangement $\mathcal{B}$ obtained by adding all the regions bordering the degeneracy to the first arrangement we will obtain the arrangement described in Subsection~\ref{subsection:1} and its leftover matrix will have entries $$L^n_{f,j}=V(\mathcal{B})^{(k)}_{k+f,k+j}=V(\mathcal{B})_{k+f,k+j}\cdot \phi (\prod^k_{i=1} (1-l_i^2(k+f,k+j))),$$ which are actually the same.

To continue the proof we need the following lemma.
\begin{lemma}\label{lemma:8}
Consider any three regions $a$, $b$, and $c$ in the second subarrangement. If $a$ and $b$ are in the same cone formed by lines passing through the degeneracy, then  $$\phi (\prod^k_{i=1} (1-l_i^2(a,c)))=\phi (\prod^k_{i=1} (1-l_i^2(b,c))).$$
\end{lemma}
\begin{proof}
We will assume that the regions $R_a$ and $R_c$ border the degeneracy point. For three arbitrary regions $a$, $b$ and $c$, we can apply this proof to the triples of regions $(a_0, a, c_0)$, $(a_0,b,c_0)$ and $(c_0,c,a_0)$, where $a_0$ is the region lying in the same cone formed by lines passing through the degeneracy as $a$, and $c_0$ is the region lying in the same cone formed by lines passing through the degeneracy as $c$, and obtain the desired result.
\begin{figure}[h]
    \centering
    \includegraphics[scale=0.6]{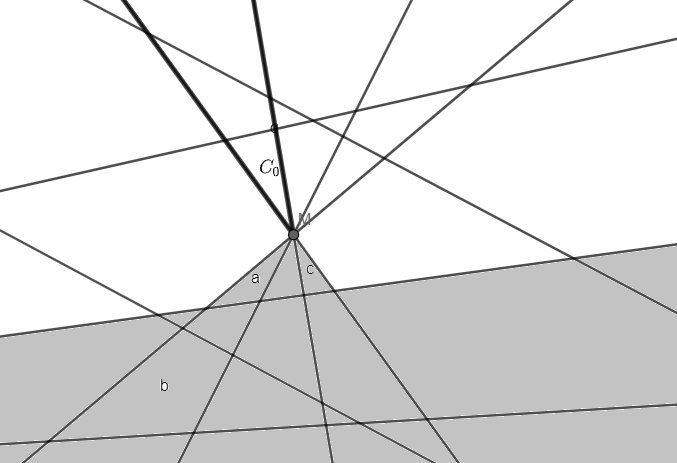}
    \caption{Illustration for Lemma~\ref{lemma:8}.}
    \label{fig:my_label8}
\end{figure}

Consider all the cones formed by lines passing through the degeneracy that are in the first subarrangement. Fix one such cone $C_0$, as is labeled in Figure~\ref{fig:my_label8}. Note that for any region in $C_0$, the lines passing through the degeneracy that separate it from both $R_a$ and $R_c$ are the same. Hence for every region $R_i$ in $C_0$, we have $1-l_i^2(a,c)=1-U_i^2P_{C_0}^2$, where $U_i$ is the product of some indeterminates of lines not passing through the degeneracy. Notice that for the region $R_d$ in $C_0$ bordering the degeneracy, we have $1-l_i^2(a,c)=1-P_{C_0}^2$, thus by Lemma~\ref{lemma:7}, $\phi(\prod_{i : R_i \in C_0} (1-l_i^2(a,c)))=1-P_{C_0}^2$. Applying the same argument to every cone formed by lines passing through a degeneracy, we deduce that the product $\phi (\prod^k_{i=1} (1-l_i^2(a,c)))$ concerns only the weights of lines passing through the degeneracy. Since there is no such line between $a$ and $b$,  we are done.
\end{proof}

Next, we want to turn all the entries $S_{i,j}$ with $i<n$ and $j \geq n$ into zero. Observe that for all $g<n$ $\dfrac{S_{i,j}}{S_{g,j}}=\dfrac{S_{i,t}}{S_{g,t}}$ for some $1\leq t \leq n-1$, so we can multiply the $t$-th column by some polynomial and subtract it from the $j$-th row in order to turn the latter into zero. This is because by Lemma~\ref{lemma:8}, $k+t$ is the index of the region bordering the degeneracy point that lies in the same part of the plane bounded by lines passing through $M$ as the $k+j$-th region.

Then we carry out the same operations on the entries $S_{i,j}$ with $i>n$ and $j \leq n$ and examine the entries $S_{i,j}$ with $i,j \geq n$ after the operations. Denote by $S'$ the matrix obtained by elimination of the first $n-1$ rows and columns of $S$ after performing the operations described above. Note that $$S'_{i,j}=S_{n+i-1,n+j-1}-\dfrac{S_{n+i-1,q}}{S_{m,q}}\cdot S_{m,n+j-1}$$ for some $m,q \in \{1,2,\ldots,n-1\}$. Set $p=n+i-1$ and $s=n+j-1$. Consider the regions $k+p$ and $k+s$. It follows from Lemma~\ref{lemma:6} that $$S_{p,s}=V_{k+p,k+s} \cdot \phi(\prod_{i=1}^{k} (1-l_i^2(k+p,k+s))).$$ Note that by Lemma~\ref{lemma:8} $$\phi(\prod_{i=1}^{k} (1-l_i^2(k+p,k+q)))=\phi(\prod_{i=1}^{k} (1-l_i^2(k+m,k+q)))=\phi(\prod_{i=1}^{k} (1-l_i^2(k+m,k+s))).$$ Denote by $Y$ the product of the weights of lines that separate both $(k+p,k+m)$ and $(k+s,k+q)$, by $Z$ the product of the weights of lines that separate $(k+p,k+m)$ but not $(k+s,k+q)$, and by $W$ the product of the weights of lines that separate $(k+s,k+q)$ but not $(k+p,k+m)$. It follows from Lemma~\ref{lemma:8} that $\dfrac{S_{p,q}}{S_{m,q}}=Y\cdot Z$ and $$S_{m,s}=Y\cdot W\cdot V_{k+m,k+q}\cdot\phi(\prod_{i=1}^{k} (1-l_i^2(k+m,k+s))),$$ 
$$S_{p,s}=Z\cdot W\cdot \phi(\prod_{i=1}^{k} (1-l_i^2(k+p,k+s))).$$ If we write $\prod_{i=1}^{k} (1-l_i^2(k+m,k+s))$ in the form $1+P$ for some polynomial $P$, then we can write $\phi(\prod_{i=1}^{k} (1-l_i^2(k+p,k+s)))$ in the form $1+Y^2P$, so $$S'_{i,j}=Z\cdot W\cdot V_{k+m,k+q}\cdot (\phi(1+Y^2P)-Y^2\phi(1+P)).$$ Let $P=P_1+P_2+...P_h$, where the $P_i$-s are monomials. Note that $\gcd(Y,P_i)=1$ because $P_i$ involves only the weights of the lines passing through the degeneracy. Then $$\phi(1+Y^2P)-Y^2\phi(1+P)=1+Y^2(\sum_{i=1}^h \phi(P_i))-Y^2-Y^2(\sum_{i=1}^h \phi(P_i))=1-Y^2.$$

\

On the other hand, consider the arrangement $\mathcal{A}'$ numbered as described above. Denote its Varchenko matrix by $U=U_{i,j}$. Let the number of edge regions be $e$. Consider the matrix formed by deletion of the first $e$ rows and columns of $U^{(e)}$. Denote it by $U'$. We show that $U'=S'$. By Lemma~\ref{lemma:6} we know that $U'_{h,j}=U_{e+h,e+j}\cdot \phi \big(\prod^e_{i=1} (1-l^2_i(e+h,e+j))\big)$, but we know that for every $a \in \{1,2,\ldots,e\}$,  $l^2_i(e+h,e+j)=1-T_a^2Y^2$ for some monomial $T_a$ and $T_m=1$. Hence $$\phi (\prod^e_{i=1} (1-l^2_i(e+h,e+j))=\phi\Big((1-Y^2)\cdot\prod_{a\in \{1,2,\ldots,m-1,m+1,\ldots,e\}} (1-T_a^2Y^2)\Big)=1-Y^2$$ by Lemma~\ref{lemma:7} and $U_{e+h ,e+j}=Z\cdot W\cdot V_{k+m,k+q}$ by definition of $Z$ and $W$. It follows that $U'_{i,j}=S'_{i,j}$ and thus $U'=S'$. Since $U'$ is diagonalizable over $\mathbb{Z}[x_1,x_2,\ldots,x_n]$, so is $S'$. 

Therefore we conclude that the resulting matrix is of the desired block diagonal form.
\end{proof}

\section{Arrangements with multiple degenerate points}\label{section 5}
In this section, we generalize Theorem~\ref{theorem1degeneracy} to pseudoline arrangements with more than one degenerate points, which is the main theorem of this paper. (See Theorem \ref{theorem45}.)
\begin{theorem}
Let $\mathcal{A}$ be a pseudoline arrangement. The Varchenko matrix of $\mathcal{A}$ has a block diagonal form, whereby the degenerate points correspond to their leftover matrices and non-degenerate elements of the intersection poset $L(\mathcal{A})$correspond to the entries of the form $$\prod_{a \in B}(1-x_a^2)$$ ranging over all $B\subseteq I$ such that $H_B \in L(\mathcal{A})$.
\end{theorem}

The main idea of the proof is to "resolve" all degenerate points by further splitting the arrangement into nondegenerate subarrangements. In order to do so, we will use a well-studied representation of pseudoline arrangements, called \textit{wiring diagrams}. A wiring diagram is an arrangement of piecewise linear pseudolines called \textit{wires}. The wires are horizontal except in small neighborhoods of their intersections with other wires. An example of a wiring diagram is shown in Figure~\ref{fig:wiring}.

\begin{figure}[h]
    \centering
    \includegraphics[scale=0.3]{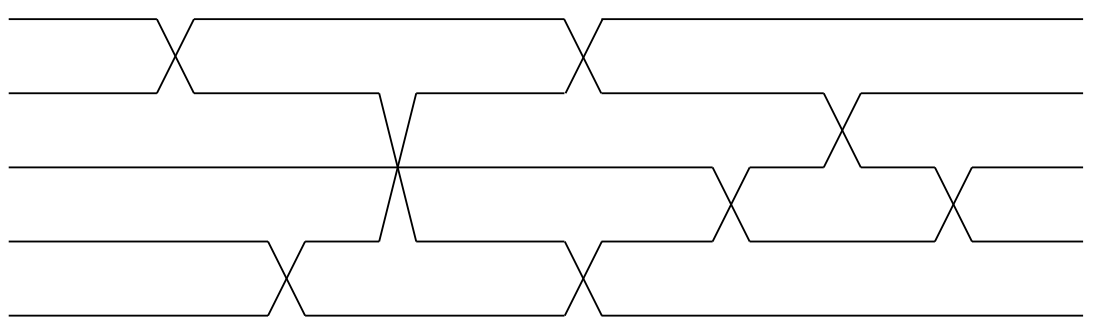}
    \caption{Wiring diagram.}
    \label{fig:wiring}
\end{figure}

The correspondence between wiring diagrams and pseudoline arrangements is captured by the following theorem. See, for instance, \cite{bjorner1999oriented}.
\begin{theorem}[Goodman, Pollack] Every arrangement of
pseudolines is isomorphic to a wiring diagram arrangement. 
\end{theorem}
Now we provide an algorithm for partitioning pseudoline arrangements.
\begin{lemma}\label{splittinglemma}
For any pseudoline arrangement $\mathcal{A}$ there is a degeneracy point $M$ with the following property: suppose that there are $n$ pseudolines passing through it, then we can divide the plane into two parts by segments of the pseudolines of $\mathcal{A}$ so that
\begin{enumerate}[(a)]
    \item $M$ lies on the border separating the two parts, and $M$ is the only degeneracy point on the border;
    \item every pseudoline of $\mathcal{A}$ intersects the border at most once;
    \item one of the resulting two subarrangements is in semigeneral position and contains $n-1$ regions that have $M$ as a vertex.
\end{enumerate}
\end{lemma}
\begin{proof}
Fix a wiring diagram representation of the pseudoline arrangement $\mathcal{A}$ and a coordinate system such that the wires are parallel to the $x$-axis. Consider a degenerate point $M$ with the biggest $x$-coordinate. Suppose that it has coordinates $(x_M;y_M)$ and it is $n$-basic. Consider the line $x=x_M$ and the half-plane $H$ bounded by it from the left. Notice that in $H$, all $n$ wires passing through $M$ are present (in the form of rays), and they form a cone $C$ consisting of exactly $n-1$ regions with $M$ as a vertex. On the other hand, as all the points in the half-plane $H$ have an $x$-coordinate bigger than $x_M$, its interior contains no degeneracy points.

Now, we look at the cone $C$, whose interior contains no degeneracies. If no wire intersects this cone at two points, we use the borders of this cone to split the arrangement. Else, let $k_1$, $k_2 \ldots k_p$ be the wires intersecting the cone in two points such that each has an intersection with the line $x=x_M$ with $y$-coordinate less than $y_M$. Similarly, let $l_1$, $l_2\ldots l_q$ be the wires that intersect the cone in two points and the line $x=x_M$ at some $y$-coordinate bigger than $y_M$. Let us call the border of the cone with the positive slope in the neighborhood of $M$ the \textit{upper border} and the one with a negative slope in the neighborhood of $M$ the \textit{lower border}. 

Consider all the points of intersections of the wires $l_i$ with the lower border. Take the one that is closest to $M$. Let this point be $P_1$ and without loss of generality, assume that it is the intersection of $l_1$ and the lower border. Now consider all the intersection points of $l_1$ and the other $l_i$'s that lie outside of the cone. Take the intersection that is the closest to $P_1$, and denote it by $P_2$. Continue the procedure in analogous way so that we get a finite sequence of points $P_1$, $P_2 \ldots P_s$.

Similarly, consider all the intersection points of the wires $k_i$ with the upper border. Take the one that is the closest to $M$. Let this point be $Q_1$ and assume it is the intersection of $k_1$ and upper border. Now consider all the intersection points of $k_1$ and the other $k_i$'s that lie outside of the cone. Take the point closest to $Q_1$ and call it $Q_2$. Analogously continue the procedure so that we get a finite sequence of points $Q_1$, $Q_2 \ldots Q_t$.
Assume without loss of generality that $P_s$ is the intersection of $l_{s-1}$ and $l_s$, and $Q_s$ is the intersection of $k_{t-1}$ and $k_t$. 

Now we define the border of the splitting as follows. It consists of the ray that lays on $l_s$, starts at $P_s$ and never intersects the cone, the piecewise linear segment $P_sP_{s-1}\ldots P_1MQ_1\ldots Q_t$, and the ray that lays on $k_t$, starts at $Q_t$ and never intersects the cone. It is easy to see that no wire intersects this border more than once; moreover, one of the half-planes obtained by this splitting, namely the one containing the cone $C$, lies solely on the right side of the line $x=x_M$, so there are no degeneracies in its closure other than $M$, and it contains $n-1$ regions that border $M$.
\end{proof}

With the above algorithm, the proof of Theorem~\ref{theorem45} immediately follows by induction. The base case is simply Theorem~\ref{theorem1degeneracy}. At the $n$th step, we resolve a new degenerate point by splitting the arrangement using Lemma~\ref{splittinglemma}. Then we apply the proof of Theorem~\ref{theorem1degeneracy}, with the first subarrangement containing the $n-1$ resolved degeneracies and the second subarrangement in semigeneral position. 

\

As an immediate application, we deduce the determinant formula for the Varchenko matrix of a pseudoline arrangement, which was proved in \cite{hochstattler2018varchenko} for all dimensions using matroidal methods.
\begin{corollary}
Let $\mathcal{A}$ be a pseudoline arrangement and $V$ its Varchenko matrix. Then $$\det(V)=\prod_{M\in L(\mathcal{A})}(1-x_M)^{l(M)},$$ where 
\begin{equation*}
l(M)=\begin{cases}
\mathrm{number\ of\ elements\ of\ the\ intersection\ poset}\ H_B \subseteq M, if\ M\ is\ a\ pseudoline; \\
0,\ if\ M\ is\ a\ nondegenerate\ intersection; \\
n-2,\ if\ M\ is\ a\ degenerate\ intersection\ of\ n\ pseudolines.
\end{cases}
\end{equation*}
\end{corollary}

\section{Future directions}
It is natural to hope that Theorem~\ref{theorem45} can be generalized to higher dimensions in the following way.
\begin{conjecture}
The Varchenko matrix of an oriented matroid has a block diagonal form, whereby non-degenerate elements of intersection poset correspond to the diagonal entries and degenerate points correspond to their leftover metrices.
\end{conjecture}

\section{Acknowledgements}
This project was conducted as part of the Research Science Institute Program (RSI) at MIT in the summer of 2019. The authors would like to thank RSI and the Department of Mathematics at MIT for the opportunity and the support throughout the project.

\end{document}